\documentclass[12pt]{amsart}

\usepackage{amsmath}
\usepackage{amssymb}
\usepackage{amscd}
\usepackage[cmtip,all]{xy}

\title{Quasi-Galois points}
\author{Satoru Fukasawa, Kei Miura and Takeshi Takahashi}

\address{Department of Mathematical Sciences, Faculty of Science, Yamagata University,  
Kojirakawa-machi 1-4-12, Yamagata 990-8560, Japan}
\email{s.fukasawa@sci.kj.yamagata-u.ac.jp} 
\thanks{The first author was partially supported by JSPS KAKENHI Grant Number 25800002.} 
\address{Department of Mathematics, Ube National College of Technology, Ube, Yamaguchi 755-8555, Japan}
\email{kmiura@ube-k.ac.jp}
\thanks{The second author was partially supported by JSPS KAKENHI Grant Number 26400057.} 
\address{Department of Information Engineering, Faculty of Engineering, Niigata University, Niigata 950-2181, Japan}
\email{takeshi@ie.niigata-u.ac.jp}
\thanks{The third author was partially supported by JSPS KAKENHI Grant Number 25400059.}

\newtheorem{theorem}{Theorem}[section]
\newtheorem{proposition}[theorem]{Proposition}
\newtheorem{corollary}[theorem]{Corollary}
\newtheorem{lemma}[theorem]{Lemma} 

\newtheorem{fact}[theorem]{Fact}

\theoremstyle{definition}
\newtheorem{example}[theorem]{Example}
\newtheorem{remark}[theorem]{Remark}
\newtheorem{definition}[theorem]{Definition}

\begin{document}
\begin{abstract} 
We introduce the new notion of the ``quasi-Galois point'' in Algebraic geometry, which is a generalization of the Galois point. 
A point $P$ in projective plane is said to be quasi-Galois for a plane curve if the curve admits a non-trivial birational transformation which preserves the fibers of the projection $\pi_P$ from $P$. 
We discuss the standard form of the defining equation of curves with quasi-Galois points, the number of quasi-Galois points, the structure of the Galois group for the projection, relations with dual curves, and so on. 
Our theory also has applications to the study of automorphism groups of algebraic curves. 
\end{abstract}
\maketitle

\section{Introduction}  
The purpose of this article is to establish the theory of the ``{\it quasi-Galois point}'' in Algebraic geometry. 

Let $K$ be an algebraically closed field of characteristic $p \ge 0$, and let $C \subset \mathbb{P}^{2}$ be an irreducible plane curve of degree $d \ge 4$. 
Consider a point $P \in \mathbb{P}^2$ and let $\pi_P: C \dashrightarrow \mathbb{P}^1$ be the projection from $P$. 
The Galois closure of $K(C)/\pi_P^*K(\mathbb P^1)$ is denoted by $L_P$, and the Galois group by $G_P$.
In this situation, Hisao Yoshihara introduced the notion of the {\it Galois point} in 1996 (see e.g. \cite{fukasawa, miura-yoshihara1, yoshihara1}). 
A point $P$ is said to be Galois, if the function field extension $K(C)/\pi_P^*K(\mathbb{P}^1)$ is Galois. 
There are many results on Galois point: the number of Galois points for smooth or singular curves, characterization of curves with Galois points, the structure of the Galois group, and so on. 
Also, Galois points have some applications to the study of algebraic curves (see \cite{fukasawa} and references therein). 
On the other hand, when $p=0$, it follows from theorems of Yoshihara \cite{yoshihara1} (if $C$ is smooth) and Pirola--Schlesinger \cite{pirola-schlesinger} that there exist only finitely many points $P \in C$ (resp. $P \in \mathbb{P}^2 \setminus C$) such that $G_P$ is not the full symmetric group $S_{d-1}$ (resp. $S_d$).  
There are several papers studying those exceptional non-Galois points (see e.g. \cite{miura1, miura-yoshihara1, miura-yoshihara2, takahashi, yoshihara1}). 

We would like to extend the study of Galois points and to study systematically points $P$ such that $G_P$ is not the full symmetric group. 
To do these, we introduce the notion of the {\it quasi-Galois point}.  
For a point $P \in \mathbb P^2$, we define the set
$$G[P]:=\{ \tau \in {\rm Bir}(C) \ | \ \tau (C \cap \ell \setminus \{P\}) \subset \ell \ \mbox{ for a general line } \ell \ni P \} $$
of all birational transformations of $C$ preserving the fibers of the projection $\pi_P$. 

\begin{definition}
If $|G[P]| \ge 2$, then we say that $P$ is a {\it quasi-Galois point}. 
\end{definition}

\begin{remark}
\begin{itemize}
\item[(1)] We have $G[P]=\{\tau \in {\rm Bir}(C) \ | \ \pi_P \circ \tau=\pi_P\}$. 
\item[(2)] The condition $|G[P]|=\deg \pi_P$ holds if and only if $P$ is Galois. 
\end{itemize} 
\end{remark}

As an extension of the study of Galois points, it is important to determine the number of quasi-Galois points, since we have succeeded in obtaining several characterization results by the number of Galois points. 
We will discuss the number of quasi-Galois points in Sections 4, 5 and 6. 
In Section 4, we give upper bounds for the number in general, and determine the number for curves of high degree (Theorems \ref{inner} and \ref{outer}).
Furthermore, we characterize smooth sextic curves attaining an upper bound (Theorem \ref{sextic}).  
In Section 5, we determine completely for Fermat curves (Theorem \ref{Fermat, number}).  
In Section 6, for quartic curves, we give a sharp upper bound, and describe curves attaining the bound (Theorem \ref{quartic}). 
Theorem \ref{quartic} presents a characterization of the Klein quartic, since it is known that the described curve is projectively equivalent to the Klein quartic. 

On the structure of $G_P$ for a quasi-Galois point $P$, we will show that $G_P$ is very small compared to the full symmetric group, at least when $p=0$ and $C$ is smooth (Proposition \ref{Galois closure}).  
In Section 7, we give a criterion for $G_P \cong (\mathbb Z/n\mathbb Z) \times D_{2n}$ in terms of quasi-Galois point (Theorem \ref{2p^2}). 
As initiated in Fukasawa--Miura's papers \cite{fukasawa-miura1, fukasawa-miura2}, by considering the dual curve, we have many examples of projections from points $P$ such that $G_P$ is not the full symmetric group.  
In Section 8, we mention the relations between quasi-Galois points and the dual curve.

Quasi-Galois points contribute to large automorphism groups. 
We will show that if the automorphism group ${\rm Aut}(C)$ of a smooth plane curve $C$ is simple and of even order, then ${\rm Aut}(C)$ is generated by associated groups $G[P]$ with quasi-Galois points $P$ (Theorem \ref{simple}). 
The Klein quartic and the Wiman sextic are such examples of curves. 
According to a theorem of Harui \cite[Theorem 2.3]{harui}, smooth plane curves with $|{\rm Aut}(C)|>6d^2$ must be the Klein quartic or the Wiman sextic.  
On the other hand, we will obtain the standard form of the defining equation of curves with quasi-Galois points (Theorem \ref{standard form}).
By considering the standard form, we have that there exist quasi-Galois points for smooth plane curves of four types (i), (iii), (iv) and (v) with large automorphism groups in Harui's classification list \cite[Theorem 2.5]{harui}. 

This article contains thirteen new theorems. 
Theorems \ref{standard form}, \ref{inner}, \ref{outer} and \ref{d/2} are generalizations of Yoshihara--Miura's theorems \cite{miura-yoshihara1, yoshihara1} on Galois point. 
Theorem \ref{sextic} is a new characterization result, by virtue of the generalization as quasi-Galois points. 
Theorems \ref{generated}, \ref{simple}, \ref{sextic, generated} and \ref{Fermat, generated} are extensions of works of Kanazawa--Takahashi--Yoshihara \cite{kty} and Miura--Ohbuchi \cite{miura-ohbuchi}.  
Theorem \ref{Fermat, number} is a generalization of Yoshihara--Miura's result \cite{miura-yoshihara1, miura-yoshihara2} on Fermat curves. 
Theorem \ref{quartic} is the case of outer points compared to Takahashi's work \cite{takahashi}. 
Theorem \ref{2p^2} is a new result on the Galois closure of a projection for a smooth plane curve (cf. \cite{miura1, miura-yoshihara1, yoshihara1}). 
Theorem \ref{dual} adds new information to Fukasawa--Miura's work \cite{fukasawa-miura1, fukasawa-miura2}. 

Needless to say, the notions of Galois points for a hypersurface \cite{yoshihara2}, Galois lines for a space curve \cite{yoshihara3}, Galois subspaces, and Galois embeddings of an algebraic variety \cite{yoshihara4} are naturally generalized as {\it quasi-Galois points}, {\it quasi-Galois lines}, {\it quasi-Galois subspaces}, and {\it quasi-Galois embeddings}.  

\section{Preliminaries} 
We introduce the system $(X:Y:Z)$ of homogeneous coordinates on $\mathbb P^2$ with local coordinates $x=X/Z, y=Y/Z$ for the affine open set $Z \ne 0$. 
If $P \in C$ is a smooth point, the (projective) tangent line at $P$ is denoted by $T_PC$. 
For a projective line $\ell \subset \mathbb P^2$ and a point $P \in C \cap \ell$, the intersection multiplicity of $C$ and $\ell$ at $P$ is denoted by $I_P(C, \ell)$.  
The line passing through points $P$ and $Q$ is denoted by $\overline{PQ}$, when $P \ne Q$, and the projection from a point $P \in \mathbb P^2$ by $\pi_P$, which is the rational map from $C$ to $\mathbb P^1$ represented by $Q \mapsto \overline{PQ}$. 
If $Q \in C$ is a smooth point, the ramification index of $\pi_P$ at $Q$ is denoted by $e_Q$.   
We note the following elementary fact.  
 
\begin{fact} \label{index}
Let $P \in \mathbb P^2$, and let $Q \in C$ be a smooth point. 
Then, for $\pi_P$ we have the following.
\begin{itemize}
\item[(1)] If $P=Q$, then $e_P=I_P(C, T_PC)-1$.  
\item[(2)] If $P \ne Q$, then $e_Q=I_Q(C, \overline{PQ})$.   
\end{itemize} 
\end{fact}
If a Galois covering $\theta:C \rightarrow C'$ between smooth curves is given, then the Galois group $G$ acts on $C$ naturally. 
The stabilizer subgroup of $P$ is denoted by $G(P)$. 
The following fact is useful (see \cite[III. 7.2, 8.2]{stichtenoth}). 
\begin{fact} \label{Galois covering} 
Let $\theta: C \rightarrow C'$ be a Galois covering of degree $d$, and let $G$ be the Galois group. 
Then, we have the following. 
\begin{itemize}
\item[(1)] The order of $G(P)$ is equal to $e_P$ at $P$ for any point $P \in C$.\item[(2)] Let $P, Q \in C$. If $\theta(P)=\theta(Q)$, then $e_P=e_Q$. 
\end{itemize} 
\end{fact}

\section{Fundamental results}

We start with the following lemma, due to field theory. 
\begin{lemma} \label{subgroup} 
The set $G[P]$ is a subgroup of ${\rm Bir}(C)$ such that the order $|G[P]|$ divides the degree of $\pi_P$. 
\end{lemma}

\begin{proof} 
It is not difficult to check that $G[P]$ is a subgroup of ${\rm Bir}(C)$. 
Since the fixed field $K(C)^{G[P]}$ is an intermediate field of $K(C)/\pi_P^*K(\mathbb P^1)$, $|G[P]|$ divides $\deg \pi_P$. 
\end{proof}

\begin{example}
Let $p=0$ and $\deg \pi_P=4$. 
Then, $P$ is quasi-Galois with $|G[P]|=2$ if and only if $G_P$ is isomorphic to the dihedral group $D_8$ of order eight.
\end{example}

Let $G_0[P] \subset G[P]$ be the set of all elements of $G[P]$ which are the restrictions of some linear transformations of $\mathbb P^2$. 

\begin{definition}
A point $P$ is said to be extendable quasi-Galois if $|G_0[P]| \ge 2$. 
\end{definition}

\begin{remark}
If $C$ is smooth, then any automorphism is the restriction of a linear transformation (see \cite[Appendix A, 17 and 18]{acgh} or \cite{chang}). 
Therefore, every quasi-Galois point is extendable and $G[P]=G_0[P]$.  
\end{remark}

\begin{theorem}[cf. \cite{miura2, yoshihara1, yoshihara5}] \label{standard form} 
If $p=0$ or $p$ does not divide the order $|G_0[P]|$, then the group $G_0[P]$ is a cyclic group.  
Furthermore, for an integer $n \ge 2$, $n$ divides $|G_0[P]|$ if and only if there exists a linear transformation $\phi$ such that 
\begin{itemize}
\item[(1)] $\phi(P)=(1:0:0)$,  
\item[(2)] there exists an element $\sigma \in G_0[\phi(P)] \subset {\rm Bir}(\phi(C))$ which is represented by the matrix 
$$ A_{\sigma}=\left(\begin{array}{ccc} \zeta & 0 & 0 \\ 0 & 1 & 0 \\
0 & 0 & 1 \end{array} \right), $$
where $\zeta$ is a primitive $n$-th roof of unity, and 
\item[(3)] $\phi(C)$ is given by 
$$ \sum_i G_{d-n i}(Y, Z)X^{n i}=0, $$
where $G_{d-n i}$ is a homogeneous polynomial of degree $d-n i$ in variables $Y, Z$. 
\end{itemize}
\end{theorem}

\begin{proof} 
We can assume that $P=(1:0:0)$. 
The projection $\pi_P$ is given by $(x:y:1) \mapsto (y:1)$. 
We have a field extension $K(x, y)/K(y)$. 
Let $\sigma \in G_0[P]$, and let $A_{\sigma}=(a_{ij})$ be a matrix representing $\sigma$, as $\sigma: (X:Y:Z) \mapsto (X:Y:Z)\ ^tA_{\sigma}$. 
Since $\sigma^*(y)=y$, 
$$ (a_{21}x+a_{22}y+a_{23})-(a_{31}x+a_{32}y+a_{33})y=0 $$ 
in $K(C)=K(x,y)$. 
Since $d \ge 4$, we have $a_{21}=a_{23}=a_{31}=a_{32}=0$ and $a_{22}=a_{33}$. 
We take the representative matrix with $a_{22}=a_{33}=1$. 
Let $N=|G_0[P]|$. 
If $a_{11}=1$, then, by $A_{\sigma}^N=1$, we have $N a_{12}=N a_{13}=0$. 
Since $N$ is not divisible by $p$ if $p>0$, $a_{12}=a_{13}=0$. 
Then, we have an injective homomorphism 
$$ G_0[P] \hookrightarrow K\setminus 0; \ \sigma \mapsto a_{11}(\sigma), $$
where $a_{11}(\sigma)$ is the $(1,1)$-element of $A_{\sigma}$. 
Therefore, $G_0[P]$ is a cyclic group.  

By the assumption, the order of any element of $|G_0[P]|$ is not divisible by $p$ if $p>0$. 
Assume that $n \ge 2$ divides $|G_0[P]|$. 
Since $G_0[P]$ is a cyclic group, there exists an element $\sigma \in G_0[P]$ of order $n$ and $\sigma$ is represented by the matrix 
$$ A_{\sigma}=\left(\begin{array}{ccc} \zeta & a & b \\ 0 & 1 & 0 \\ 0 & 0 & 1 \end{array}\right), $$
where $a, b \in K$ and $\zeta$ is a primitive $n$-th root of unity. 
If we take 
$$ B=\left(\begin{array}{ccc} 1 & a & b \\ 0 & 1-\zeta & 0 \\ 0 & 0 & 1-\zeta \end{array}\right), $$
then 
$$ B^{-1}A_{\sigma}B=\left(\begin{array}{ccc} \zeta & 0 & 0 \\ 0 & 1 & 0 \\ 0 & 0 & 1 \end{array}\right).  $$
We take the linear transformation given by $(X:Y:Z) \mapsto (X:Y:Z) \ ^tB^{-1}$, so that we have assertion (2). 
Let $f(x,y)=\sum_{i}a_i(y)x^i$ be a defining polynomial with $a_0(y) \ne 0$. 
Then, $\sigma^*f=\sum_ia_i(y)\zeta^ix^i$. 
There exists $c \in K$ such that $cf=\sigma^*f$. 
Since $a_0(y) \ne 0$, we have $c=1$. 
For all $i$, $a_i(y)=a_i(y)\zeta^i$. 
If $a_i(y) \ne 0$, then $\zeta^i=1$. 
This implies that $n$ divides $i$. 
We have assertion (3). 

The if-part is obvious. 
\end{proof} 

\begin{corollary} \label{fixed locus} 
For $\sigma \in G_0[P] \setminus \{1\}$, we define $F[P]:=\{Q \in \mathbb P^2 \ | \  \sigma(Q)=Q \}$. 
If we use the standard form as in Theorem \ref{standard form}, $F[P]=\{P\} \cup \{X=0\}$. 
In particular, the set $F[P]$ does not depend on $\sigma$. 
\end{corollary} 

\begin{corollary} \label{two groups}
Let $P_1, P_2 \in \mathbb{P}^2$. 
If $P_1 \ne P_2$, then $G_0[P_1] \cap G_0[P_2]=\{1\}$. 
\end{corollary}

\begin{example}[Miura, Remark 1 in \cite{yoshihara1}] 
If we do not assume that $C$ is smooth, there exists an example with $1 < |G_0[P]| < |G[P]|$. 
For example, if $p=0$, $C$ is defined by
$$ Y(X^2+Y^2)^n+X^{n+1}Z^n+Y^{n+1}Z^n+YZ^{2n}=0, $$
and $P=(0:0:1)$, then $|G_0[P]|=n$, and $G[P]=G_P$ is the dihedral group of order $2n$. 
\end{example}

\begin{proposition} \label{Galois closure}
Assume that $p=0$ and $P=(1:0:0)$ is extendable quasi-Galois with $|G_0[P]|=n \ge 2$. 
Then, $C$ is given by $f(x^n, y)=0$  in the affine plane $Z \ne 0$ for some polynomial $f(X, y)$. 
Let $r$ be the degree of $f(X, y)$ in variable $X$, and let $R$ be the degree of the Galois closure of $K(X, y)/K(y)$. 
Then, $$ n \times r \le |G_P| \le R \times n^r. $$
In particular, 
$$ n \le |G_P| \le r! \times n^r. $$
\end{proposition}

\begin{proof}
Let $x_1(=x^n), x_2, \ldots, x_r$ be the roots of $f(X, y)$. 
Then, we have the following diagram.  
\begin{equation*} 
\xymatrix{K(x,y) \ar@{-}[r] \ar@{-}[d] & K(y, x, \sqrt[n]{x_2}, \ldots, \sqrt[n]{x_r}) \ar@{-}[d] \\
K(x^n, y) \ar@{-}[d] \ar@{-}[r] & K(y, x^n, x_2, \ldots, x_r) \ar@{-}[ld] \\
K(y) & }
\end{equation*} 
Then, the field $K(y, x, \sqrt[n]{x_2}, \ldots, \sqrt[n]{x_r})$ becomes the Galois closure of $K(x,y)/K(y)$. 
Since $[K(x,y):K(y)]=n \times r$ and $[K(y, x, \sqrt[n]{x_2}, \ldots, \sqrt[n]{x_r}):K(y, x^n, x_2, \ldots, x_r)] \le n^r$, we have the conclusion. 
\end{proof}

If $C$ is smooth, we define 
$$G_n(C)=\langle G[P] \ | \ P: \mbox{ quasi-Galois with } |G[P]|=n \rangle \subset {\rm Aut}(C) $$
after Kanazawa--Takahashi--Yoshihara \cite{kty} and Miura--Ohbuchi \cite{miura-ohbuchi}. 
Similar to \cite[Theorem 1]{miura-ohbuchi}, we have the following. 

\begin{theorem} \label{generated}
Let $C$ be smooth. 
Then, $G_n(C)$ is a normal subgroup of ${\rm Aut}(C)$. 
\end{theorem}

On the other hand, we have the following. 

\begin{proposition} \label{involution}
Let $p \ne 2$, and let $C$ be smooth. 
The following conditions are equivalent. 
\begin{itemize}
\item[(1)] The order $|{\rm Aut}(C)|$ is even. 
\item[(2)] The group ${\rm Aut}(C)$ contains an involution. 
\item[(3)] There exists a quasi-Galois point $P$ such that $|G[P]|$ is even. 
\end{itemize} 
\end{proposition}

\begin{proof} 
If $|{\rm Aut}(C)|$ is even, then a Sylow $2$-group contains an element of order two. 
Therefore, we have (1) $\Rightarrow$ (2).
Since $G[P]$ is a subgroup of ${\rm Aut}(C)$ as in Lemma \ref{subgroup}, the assertion (3) $\Rightarrow$ (1) is obvious. 

We prove (2) $\Rightarrow$ (3). 
Let $\sigma \in {\rm Aut}(C)$ be an involution, and let $A_{\sigma}$ be a matrix representing $\sigma$. 
Since $A_{\sigma}^2=\lambda I_3$ for some $\lambda \ne 0$, where $I_3$ is the identity matrix, we can assume that $A_{\sigma}^2=I_3$.  
The eigenvalues of $A_{\sigma}$ are $1$ or $-1$. 
For any vector $x \in K^3$, $(A-E)x$ and $(A+E)x$ are contained in the direct sum of eigenspaces, since $(A+E)(A-E)=(A-E)(A+E)=0$. 
Then, $2x=-(A-E)x+(A+E)x$ also. 
By the assumption $p \ne 2$, the direct sum of eigenspaces spans the vector space $K^3$. 
We find that $A_{\sigma}$ is diagonalizable. 
For a suitable system of coordinates, we can assume that $A_{\sigma}$ is one of the following matrices: 
$$ \pm\left(\begin{array}{ccc} -1 & 0 & 0 \\
0 & 1 & 0 \\
0 & 0 & 1 
\end{array}\right), \ 
\pm\left(\begin{array}{ccc} 1 & 0 & 0 \\
0 & -1 & 0 \\
0 & 0 & 1 
\end{array}\right), \ 
\pm\left(\begin{array}{ccc} 1 & 0 & 0 \\
0 & 1 & 0 \\
0 & 0 & -1 
\end{array}\right). $$
It follows from Theorem \ref{standard form} that $(1:0:0)$, $(0:1:0)$ and $(0:0:1)$ are quasi-Galois, for each case. 
Then, $|G[P]|$ is even. 
\end{proof}

Combining Theorem \ref{generated} and Proposition \ref{involution}, we have the following. 

\begin{theorem} \label{simple}
Let $p \ne 2$, and let $C$ be smooth. 
If ${\rm Aut}(C)$ is simple and $|{\rm Aut}(C)|$ is even, then ${\rm Aut}(C)=G_n(C)$ for some even integer $n \ge 2$. 
For example, if $p=0$, and $C$ is the Klein quartic or the Wiman sextic (see \cite{dik,  wiman} and \cite[Remark 2.4]{harui}), then ${\rm Aut}(C)=G_2(C)$.   
\end{theorem}

\begin{proof}
By Proposition \ref{involution}, if $|{\rm Aut}(C)|$ is even, then there exists a quasi-Galois point $P$ with $|G[P]|=n$, where $n$ is an even integer. 
Since $G_n(C) \ne \{1\}$ is a normal subgroup by Theorem \ref{generated}, by the assumption that ${\rm Aut}(C)$ is simple, we have that $G_n(C)={\rm Aut}(C)$. 

For the Klein quartic and the Wiman sextic, ${\rm Aut}(C)$ is a simple group of even order (see \cite[Section 232]{burnside}, \cite[pp.348--349]{hartshorne}, \cite{klein} for the Klein quartic and see \cite{wiman} for the Wiman sextic). 
Therefore, the number $n$ is equal to two or four (resp. two or six) for the Klein quartic (resp. the Wiman sextic). 
It follows from \cite[Example 4.8]{miura-yoshihara1} that there does not exist a point $P$ with $|G[P]|=4$ for the Klein quartic. 
For the Wiman sextic, if there exists a point $P$ with $|G[P]|=6$, then there exists an element of order six, by Theorem \ref{standard form}. 
Since ${\rm Aut}(C)\cong A_6$ for the Wiman sextic, there does not exist an element of order six. 
Therefore, $n=2$.  
\end{proof}

In Sections 4, 5 and 6, we study the number of quasi-Galois points. 
To do this, we introduce some symbols here. 
The number of quasi-Galois points $P \in C$ with $|G[P]|=n$ (resp. $|G[P]| \ge n$) is denoted by $\delta[n]$ (resp. $\delta[\ge n]$). 
Similarly, we define $\delta'[n]$ and $\delta'[\ge n]$, when we consider the case $P \in \mathbb P^2 \setminus C$. 

Hereafter this article, we assume that $p=0$. 

\section{The number of quasi-Galois points for smooth curves} 
In this section, we assume that $C$ is smooth. 

Let $P \in \mathbb P^2$ be quasi-Galois for $C$ with $|G[P]|=n \ge 2$. 
Then, $G[P]=G_0[P]$. 
We consider ramification points for the projection $\pi_P$. 

\begin{proposition} \label{tangent2} 
There exist $d$ points $Q_1, \ldots, Q_d \in C \cap (F[P]\setminus \{P\})$ such that $P \in T_{Q_i}C$ and $I_{Q_i}(C, T_{Q_i}C)=l_i n$ for some integer $l_i \ge 1$. 
\end{proposition}

\begin{proof} 
Let $Q \in C \cap (F[P] \setminus \{P\})$. 
By Corollary \ref{fixed locus}, $\sigma(Q)=Q$ for each $\sigma \in G[P]$. 
By Fact \ref{Galois covering}(1), the ramification index at $Q$ for  the covering map $C \mapsto C/G[P]$ is equal to $n$. 
Since the projection $\pi_P$ is the composite map of $C \rightarrow C/G[P]$ and $C/G[P] \rightarrow \mathbb P^1$, the ramification index $e_Q$ at $Q$ for $\pi_P$ is equal to $l n$ for some $l \ge 1$. 
By Fact \ref{index}(2), $e_Q=I_Q(C, \overline{PQ})=l n$ and $\overline{PQ}=T_QC$. 
Furthermore, the line given by $F[P]\setminus \{P\}$ consists of exactly $d$ points. 
\end{proof}

If $P \in C$, we have the following. 

\begin{proposition} \label{tangent}
If $P \in C$, then we have $I_P(C, T_PC)=l n+1$ for some integer $l \ge 1$. 
\end{proposition}

\begin{proof}
By Corollary \ref{fixed locus}, for any $\sigma \in G[P]$, $\sigma(P)=P$. 
Then the covering map $C \rightarrow C/G[P]$ is ramified at $P$ with index $n$, by Fact \ref{Galois covering}(1). 
Since the projection $\pi_P$ is the composite map of $C \rightarrow C/G[P]$ and $C/G[P] \rightarrow \mathbb P^1$, the ramification index $e_P$ at $P$ is equal to $l n$ for some $l \ge 1$. 
Note that $e_P=I_P(C, T_PC)-1$, by Fact \ref{index}(1). 
We have $I_P(C, T_PC)=l n+1$. 
\end{proof}

The following fact is well-known. 

\begin{lemma} \label{total-cyclic}
Let $G \subset {\rm Aut}(C)$ be a subgroup, and let $Q \in C$ be a point.  
If $\sigma(Q)=Q$ for any $\sigma \in G$, then $G$ is a cyclic group. 
\end{lemma} 

Using this fact, we have the following. 

\begin{proposition}\label{two quasi-Galois}
Let $P_1, P_2\in \mathbb P^2$ be points with $|G[P_1]|=n_1 \ge 2$, $|G[P_2]|=n_2 \ge 2$. 
\begin{itemize}
\item[(1)] If $P_1, P_2 \in C$, then $C \cap F[P_1] \cap F[P_2] \subset \{P_1, P_2\}$. 
Furthermore, if $n_1$ and $n_2$ are not coprime, then we have $C \cap F[P_1] \cap F[P_2]=\emptyset$. 
\item[(2)] If $P_1, P_2 \in \mathbb{P}^2 \setminus C$, then $C \cap F[P_1] \cap F[P_2] = \emptyset$. 
\end{itemize} 
\end{proposition}

\begin{proof}
Assume that there exists a point $Q \in C \cap F[P_1] \cap F[P_2]$. 
Note that, by definition, points $P_1, P_2$ and $Q$ are collinear. 

First, we assume that $n_1$ and $n_2$ are divisible by some integer $n \ge 2$. 
Since $G[P_1]$ and $G[P_2]$ are cyclic by Theorem \ref{standard form}, there exist subgroups of $G[P_1]$ and $G[P_2]$ of order $n$ respectively. 
Let $G$ be the group generated by such subgroups.  
Then, $G$ fixes the point $Q$. 
By Lemma \ref{total-cyclic}, $G$ is a cyclic group. 
Therefore, by Corollary \ref{two groups}, $G$ is a cyclic group of order $n^2$. 
However, the cyclic group of order $n^2$ has a unique subgroup of order $n$. 
This is a contradiction. 
In particular, we have the latter assertion of (1). 

Next, we consider the case where $Q \ne P_1, P_2$. 
Let $\sigma \in G[P_1] \setminus \{1\}$. 
Since $\sigma$ fixes $P_1$ and $Q$ on the line $\overline{P_1Q}=\overline{P_2Q}$, we have $P_3:=\sigma(P_2) \ne P_2$. 
Then, $G[P_3]=\sigma G[P_2] \sigma^{-1}$ and $Q \in C \cap F[P_2] \cap F[P_3]$. 
By the above discussion, we have a contradiction. 
We have assertions (1) and (2). 
\end{proof} 

\begin{remark}
When $P_1, P_2 \in C$ and $|G[P_1]|$, $|G[P_2]|$ are coprime, the assertion $C \cap F[P_1] \cap F[P_2]=\emptyset$ does not hold in general. 
Let $P_1=(1:0:0)$, $P_2=(0:1:0)$, and let $C$ be the smooth plane curve defined by 
$$ X^6Z+X^3Y^4+Y^6Z+Z^7=0. $$
By Theorem \ref{standard form}, $|G[P_1]|=3$, $|G[P_2]|=2$, and $C\cap F[P_1] \cap F[P_2]=\{P_1, P_2\}$. 
\end{remark}

For the number of quasi-Galois points on $C$, we have the following.

\begin{theorem} \label{inner} Let $n:=n_1 < n_2 <\dots < n_r$ be positive integers which are divisible by some integer $m \ge 2$. Then, we have the following. 
\begin{itemize}
\item[(1)]  $\sum_{i=1}^{r} \delta[n_i](n_i-1+d(n_i-2)) \le 3d(d-2)$. 
\item[(2)] If $(n+1)(n-1+d(n-2)) > 3d(d-2)$, then $\sum_i \delta[n_i] \le 1$.  
\item[(3)] If $(n^2+n+1)(n-1+d(n-2))> 3d(d-2)$, then $\sum_i\delta[n_i] \le d$.
\end{itemize} 
\end{theorem} 

\begin{proof}
Let $P \in C$ be a point with $|G[P]|=n \ge 2$. 
By Proposition \ref{tangent}, $I_P(C, T_PC) \ge n+1$. 
By Proposition \ref{tangent2}, there exist $d$ points $Q \in C \cap (F[P] \setminus \{P\})$ such that $P \in T_QC$ and $I_Q(C, T_QC) \ge n$. 
Therefore, for each quasi-Galois point $P \in C$ with $G[P]=n$, we need at least $(d+1)$  flexes (one flex if $n=2$), and the sum of their orders is at least
$$ n+1-2+d(n-2). $$
It follows from Proposition \ref{two quasi-Galois} that there exists no point $Q \in C$ such that $Q \in F[P_1] \cap F[P_2]$ for different quasi-Galois points $P_1$ and $P_2$ if $|G[P_1]|$ and $|G[P_2]|$ are not coprime. 
By the flex formula \cite[Theorem 1.5.10]{namba}, we have assertion (1). 

We consider assertion (2). 
Assume by contradiction that $\sum_i \delta[n_i] \ge 2$. 
Let $P_1, P_2 \in C$ be points such that $|G[P_1]|$, $|G[P_2]|$ are at least $n$ and divisible by $m$. 
By Proposition \ref{two quasi-Galois}, $\sigma(P_2) \ne P_2$ for any $\sigma \in G[P_1] \setminus \{1\}$.  
Therefore, we have $\sum_i \delta[n_i] \ge n+1$. 
By assertion (1), we have the inequality 
$$ (n+1)(n-1+d(n-2)) \le 3d(d-2).  $$ 
This is a contradiction. 

We consider assertion (3). 
Assume by contradiction that $\sum_i \delta[n_i] >d$.  
Then, we have three points $P_1, P_2, P' \in C$ such that $|G[P_1]|$, $|G[P_2]|$ and $|G[P']|$ are at least $n$, are divisible by $m$, and $P' \not\in \overline{P_1P_2}$. 
Similar to the proof of assertion (2), we have at least $n+1$ quasi-Galois points $P_1, \ldots, P_{n+1}$ on the line $\overline{P_1P_2}$. 
Further, for each $i \in \{1, \ldots, n+1\}$, there exist $n+1$ quasi-Galois points on the line $\overline{P_iP'}$. 
Therefore, we have at least $(n+1)\times n+1$ quasi-Galois points. 
By assertion (1), we have the inequality 
$$ (n^2+n+1)(n-1+d(n-2)) \le 3d(d-2). $$
This is a contradiction. 
\end{proof}

\begin{corollary} 
Assume that $d \ge 7$, $d$ is odd, and $n=(d-1)/2$.
Then, we have $\delta[\ge n] \le d$. 
Furthermore, $\delta[\ge n] \le 1$ if $d \ge 15$. 
\end{corollary}

For the number of quasi-Galois points in $\mathbb P^2 \setminus C$, we have the following. 

\begin{theorem} \label{outer}
\begin{itemize}
\item[(1)] $\sum_{n \ge 3} \delta'[n](n-2) \le 3(d-2)$. 
\item[(2)] If $(n+1)(n-2) > 3(d-2)$, then $\delta'[\ge n] \le 3$.  
\end{itemize} 
\end{theorem}

\begin{proof}
We consider assertion (1). 
Similar to the proof of Theorem \ref{inner}, for each quasi-Galois point $P \in \mathbb{P}^2 \setminus C$, we need at least $d$ flexes, and the sum of their orders is at least
$$ d \times (n-2). $$
Using Proposition \ref{two quasi-Galois} and the flex formula, we have 
$$ \sum_{n \ge 3} \delta'[n] d(n-2) \le 3d(d-2). $$

We consider assertion (2). 
Assume that $\delta'[\ge n] \ge 2$. 
Let $P_1, P_2 \in \mathbb{P}^2 \setminus C$ be points with $|G[P_1]| \ge n$, $|G[P_2]| \ge n$. 
By Corollary \ref{fixed locus}, if $P_2 \not \in F[P_1]$, then $\delta[\ge n] \ge n+1$. 
By assertion (1), we have the inequality 
$$ (n+1)d(n-2) \le 3d(d-2).  $$ 
This is a contradiction. 
Therefore, $P_2 \in F[P_1]$. 
If $P_3$ is another quasi-Galois point, then $P_3 \in F[P_1] \cap F[P_2]$. 
We have $\delta'[\ge n] \le 3$. 
\end{proof}

We consider the case where $d \ge 6$, $d$ is even, and $n=d/2$. 
By Theorem \ref{outer}(1), 
$$ \delta'[\ge n] \le 3 \times \frac{d-2}{n-2} \le 3 \times 4=12. $$
Note that $(n+1)(n-2)>3(d-2)$ if and only if $d \ge 14$. 
Using Theorem \ref{outer}(2), we have the following. 

\begin{theorem} \label{d/2}
Assume that $d \ge 14$, $d$ is even, and $n=d/2$. 
Then, $\delta'[\ge n]\ge 2$ if and only if $C$ is projectively equivalent to the curve defined by 
$$ X^{2n}+Y^{2n}+Z^{2n}+aX^nY^n+bY^nZ^n+cZ^nX^n=0, $$
where $a, b, c \in K$. 
In this case, $\delta'[\ge n]=3$. 
\end{theorem}

\begin{proof}
The if part is due to Theorems \ref{standard form}. 
We prove the only-if part. 
Assume that $\delta'[\ge n] \ge 2$. 
Let $P_1, P_2 \in \mathbb{P}^2 \setminus C$ be points with $|G[P_1]|\ge n$, $G[P_2]| \ge n$. 
As in the proof of Theorem \ref{outer}, $P_2 \in F[P_1]$ and $P_1 \in F[P_2]$. 
By Theorem \ref{standard form}, for a suitable system of coordinates, we can assume that $P_1=(1:0:0)$ and there exists an element $\sigma_1 \in G[P_1]$ of order $n$ which is represented by the matrix 
$$A_{\sigma_1}=\left(\begin{array}{ccc} \zeta & 0 & 0 \\ 0 & 1 & 0 \\ 0 & 0 & 1 \end{array}\right),  
$$ 
where $\zeta$ is a primitive $n$-th root of unity. 
Then, the line given by $F[P_1]\setminus \{P_1\}$ is defined by $X=0$. 
Since $P_2 \in F[P_1] \setminus \{P_1\}$, $P_2=(0:1:a)$ for some $a \in K$. 
If we take the linear transformation $(X:Y:Z) \mapsto (X:Y:Z-aY)$, we can assume that $P_2=(0:1:0)$. 
Then, there exists an element $\sigma_2 \in G[P_2]$ of order $n$ which is represented by the matrix 
$$A_{\sigma_2}=\left(\begin{array}{ccc} 1 & 0 & 0 \\ a & \zeta & b \\ 0 & 0 & 1 \end{array}\right), $$ 
for some $a,b \in K$.  
Since the line given by $F[P_2]\setminus \{P_2\}$ is defined by $aX+(\zeta-1)Y+bZ=0$ and $P_1 \in F[P_2] \setminus \{P_2\}$, we have $a=0$. 
If we take 
$$ B=\left(\begin{array}{ccc} 1-\zeta & 0 & 0 \\ 0 & 1 & b \\ 0 & 0 & 1-\zeta \end{array}\right), $$
then 
$$ B^{-1}A_{\sigma_1}B=\left(\begin{array}{ccc} \zeta & 0 & 0 \\ 0 & 1 & 0 \\ 0 & 0 & 1 \end{array}\right), \ B^{-1}A_{\sigma_2}B=\left(\begin{array}{ccc} 1 & 0 & 0 \\ 0 & \zeta & 0 \\ 0 & 0 & 1 \end{array}\right).  $$
By taking the linear transformation represented by $B^{-1}$, we have the defining polynomial  
$$ F=X^{2n}+(aY^n+bZ^n)X^n+(\alpha Y^{2n}+\beta Y^nZ^n+\gamma Z^{2n}) $$
of $C$. 
We can assume $\alpha=\gamma=1$. 
We have the conclusion. 
By Theorems \ref{standard form} and \ref{outer}, we have $\delta'[\ge n]=3$. 
\end{proof}

\begin{corollary}
Assume that $d \ge 6$, $d$ is even, and $n=d/2$. 
Then, $\delta'[\ge n]\le 12$. 
Furthermore, if $d \ge 14$, then $\delta'[\ge n]=0, 1$ or $3$. 
\end{corollary}

We consider the case where $d=6$ and $n=3$. 
We characterize smooth plane curves $C$ of degree $d=6$ with $\delta'[3]=12$. 
To do this, we introduce the notion of ``$G$-pairs''. 
Let $P_1, P_2 \in \mathbb{P}^2 \setminus C$ be points such that $P_1 \ne P_2$ and $|G[P_1]|=|G[P_2]|=n \ge 2$.  
We call the pair $(P_1, P_2)$ a {\it $G$-pair} if $\sigma_1(P_2)=P_2$ and $\sigma_2(P_1)=P_1$ for generators $\sigma_1 \in G[P_1]$ and $\sigma_2 \in G[P_2]$. 
By Corollary \ref{fixed locus}, the definition does not depend on the choice of generators. 

\begin{lemma} \label{pair 1}
Let $P_1, P_2 \in \mathbb{P}^2 \setminus C$ such that $P_1 \ne P_2$ and $|G[P_1]|=|G[P_2]|=n \ge 2$, and let $\sigma_i \in G[P_i]$ be a generator  for $i=1, 2$. 
If $\sigma_1(P_2)=P_2$, then $\sigma_2(P_1)=P_1$. 
In particular, $(P_1, P_2)$ is a $G$-pair. 
\end{lemma} 

\begin{proof} 
By the assumption, $P_2 \in F[P_1]\setminus \{P_1\}$. 
By Corollary \ref{fixed locus} and Proposition \ref{tangent2}, we have that the set $F[P_1]\setminus \{P_1\}$ is a line containing $d$ points $Q_1, \ldots, Q_d \in C$ with $\overline{P_1Q_i}=T_{Q_i}C$ for each $i$.
Since $F[P_1] \setminus \{P_1\}$ is a line passing through $P_2$, we have $\sigma_2(Q_1)=Q_i$ and $\sigma_2(Q_2)=Q_j$ for some $i, j$. 
Since $\overline{P_1Q_1}$ and $\overline{P_1Q_i}$ are tangent lines at $Q_1$ and $Q_i$ respectively, $\sigma_2(\overline{P_1Q_1})=\overline{P_1Q_i}$. 
Then, $\sigma_2(\overline{P_1Q_1} \cap \overline{P_1Q_2}) \subset \overline{P_1Q_i} \cap \overline{P_1Q_j}=\{P_1\}$. 
We have $\sigma_2(P_1)=P_1$.  
\end{proof}

We prove the following characterization theorem. 

\begin{theorem} \label{sextic} 
Let $C \subset \mathbb{P}^2$ be a smooth plane curve of degree $d=6$. 
Then, $$ \delta'[3] \le 12.$$ 
Furthermore, the equality holds if and only if $C$ is projectively equivalent to the curve defined by 
$$ X^6+Y^6+Z^6-10(X^3Y^3+Y^3Z^3+Z^3X^3)=0. $$
\end{theorem}

\begin{proof} 
Assume that $\delta'[3]=12$. 
Let $\zeta$ be a primitive cubic root of unity. 

First, we prove that there exists a $G$-pair. 
Assume by contradiction that $\sigma(P_2) \ne P_2$ for any quasi-Galois points $P_1, P_2$ and any generator $\sigma \in G[P_1]$. 
If quasi-Galois points are not collinear, then we have $3^2+3+1=13$ quasi-Galois points, similar to the proof of Theorem \ref{inner}(3). 
If quasi-Galois points are collinear, we have $3m+1$ quasi-Galois points for some integer $m$ by the actions associated with one quasi-Galois point.
In each case, we have a contradiction.   
Therefore, we have points $P_1$ and $P_2$ such that $\sigma_1(P_2)=P_2$ for a generator $\sigma_1 \in G[P_1]$. 
By Lemma \ref{pair 1}, $(P_1, P_2)$ is a $G$-pair. 
For a suitable system of coordinates, we can assume that $P_1=(1:0:0)$ and $P_2=(0:1:0)$. 
Further, similar to the proof of Theorem \ref{d/2}, we can assume that generators $\sigma_1$ of $G[P_1]$ and $\sigma_2$ of $G[P_2]$ are given by the matrices 
$$ A_{\sigma_1}=\left(\begin{array}{ccc} 
\zeta & 0 & 0 \\
0 & 1 & 0 \\
0 & 0 & 1 \end{array}\right), \ 
A_{\sigma_2}=\left(\begin{array}{ccc} 
1 & 0 & 0 \\
0 & \zeta & 0 \\
0 & 0 & 1 \end{array}\right) $$
respectively, and $C$ is given by 
$$X^6+aY^6+bZ^6+cX^3Y^3+dY^3Z^3+eZ^3X^3=0, $$
where $a, b, c, d, e \in K$. 
Therefore, $P_3=(0:0:1)$ is also quasi-Galois. 

Next, we prove that $9$ quasi-Galois points other than $P_1, P_2$ or $P_3$ are not contained in the set $S:=\overline{P_1P_2}\cup\overline{P_2P_3}\cup\overline{P_3P_1}=\{XYZ=0\}$. 
If there exists one quasi-Galois point $\not \in S$, then there exist $9$ quasi-Galois points on $\mathbb P^2 \setminus S$, by the actions $\sigma_1$ and $\sigma_2$. 
Therefore, we assume that all quasi-Galois points are contained in $S$. 
If there exist two quasi-Galois points $P_4, P_5 \not\in \{P_1, P_2, P_3\}$ which are contained in $X=0$ and $Y=0$ respectively, then $(P_4, P_5)$ is not a $G$-pair, since $(P_4, P_1)$ is a $G$-pair, $(P_4, P_3)$ is not a $G$-pair, and $P_5 \in \overline{P_1P_3}$. 
We can find quasi-Galois points in $\mathbb{P}^2\setminus S$ by the actions associated with $P_4$. 
Therefore, we can assume that $9$ quasi-Galois points $\ne P_1, P_2, P_3$ are contained in one line $\subset S$.  
We can assume that such a line is $\overline{P_2P_3}$. 
Then, we have $11$ quasi-Galois points $P_2, \ldots, P_{12}$ on $\overline{P_2P_3}$. 
By Lemma \ref{pair 1}, $(P_i, P_1)$ is a $G$-pair for $i=2, \ldots, 12$. 
However, the number of collinear quasi-Galois points $P_i$ such that $(P_i, P_1)$ is a $G$-pair must be even. 
Indeed, for every $P_i \in \{P_2, \ldots, P_{12}\}$, there exists a unique $P_j \in \{P_2, \ldots, P_{12}\}$ such that $(P_i, P_j)$ is a $G$-pair. 
This is a contradiction. 

Let $(\alpha :\beta:1)$ be a quasi-Galois point on $\mathbb{P}^2 \setminus S$. 
By using the linear transformation given by $(X:Y:Z) \mapsto ((1/\alpha)X:(1/\beta)Y:Z)$, we can assume that $(\alpha:\beta:1)=(1:1:1)$. 
Then,  points 
$$ P_{ij}:=(\zeta^i:\zeta^j:1)$$
are quasi-Galois for $i, j=0, 1, 2$, and the set $\{P_1, P_2, P_3\} \cup \{P_{ij} \ | \ i, j=0, 1, 2\}$ consists of all quasi-Galois points for $C$. 

We compute a generator $\tau \in G[P_{00}]$, where $P_{00}=(1:1:1)$. 
We have $\tau(P_1)=P_{i0}$, $\tau(P_2)=P_{0j}$ and $\tau(P_3)=P_{kk}$ for some $i, j, k \ne 0$. 
We can assume that $k=2$ and $\tau(P_3)=(1:1:\zeta)$.
Then, $\tau$ is represented by the matrix
$$ A_{\tau}=\left( 
\begin{array}{ccc}
\lambda \zeta^i & \mu & 1 \\ 
\lambda & \mu \zeta^j & 1 \\
\lambda & \mu & \zeta 
\end{array}
\right), $$ 
for some $\lambda, \mu \in K \setminus \{0\}$. 
By using the condition $\tau((1:1:1))=(1:1:1)$, we have 
$$ \lambda=\frac{\zeta-1}{\zeta^i-1}, \ \mu=\frac{\zeta-1}{\zeta^j-1}. $$
If $i=2$, then $\lambda=1/(\zeta+1)=-\zeta$ and $\tau((1:0:1))=(0:1:0)$. 
Since $(1:0:1)$ is not quasi-Galois, this is a contradiction. 
We have $i=1$ and $\lambda=1$. 
Similarly, we have $j=1$ and $\mu=1$. 

Note that 
$$ A_{\sigma_1}A_{\tau}A_{\sigma_1}A_{\tau}=
\left(\begin{array}{ccc} 
3\zeta & 0 & 0 \\
0 & 0 & 3\zeta \\
0 & 3\zeta & 0 
\end{array}\right). $$
The linear transformation given by $(X:Y:Z) \mapsto (X:Z:Y)$ acts on $C$. 
Similarly, the linear transformation given by $(X:Y:Z) \mapsto (Z:Y:X)$ acts on $C$. 
Therefore, the defining equation of $C$ is of the form 
$$ F=X^6+Y^6+Z^6+a(X^3Z^3+Y^3Z^3+Z^3X^3)=0$$
for some $a \in K$. 
We consider the action by $\tau$. 
Polynomials $(\tau^{-1})^*F$ and $F$ are the same up to a constant. 
We consider the coefficient of $X^4YZ$. 
We have that the coefficient of $X^4YZ$ is $30\zeta$ for $(\zeta X+Y+Z)^6, (X+\zeta Y+Z)^6$ and $(X+Y+\zeta Z)^6$. 
The coefficient is $3\zeta$ for $(\zeta X+Y+Z)^3(X+\zeta Y+Z)^3$, $(X+\zeta Y+Z)^3(X+Y+\zeta Z)^3$ and $(X+Y+\zeta Z)^3(\zeta X+Y+Z)^3$. 
We have $a=-10$. 

On the contrary, we consider the curve $C$ given by 
$$ X^6+Y^6+Z^6-10(X^3Y^3+Y^3Z^3+Z^3X^3)=0. $$
By Theorem \ref{standard form}, points $P_1=(1:0:0)$, $P_2=(0:1:0)$ are quasi-Galois, and groups $G[P_1], G[P_2]$ are generated by the linear transformations $\sigma_1, \sigma_2$ given by 
$$ A_{\sigma_1}=\left(\begin{array}{ccc} 
\zeta & 0 & 0 \\
0 & 1 & 0 \\
0 & 0 & 1 \end{array}\right), \ 
A_{\sigma_2}=\left(\begin{array}{ccc} 
1 & 0 & 0 \\
0 & \zeta & 0 \\
0 & 0 & 1 \end{array}\right) $$
respectively. 
Let $\tau$ be the linear transformation given by the matrix 
$$ A_{\tau}=\left(\begin{array}{ccc}
\zeta & 1 & 1 \\
1 & \zeta & 1 \\
1 & 1 & \zeta 
\end{array} \right). $$
Then, $\tau(C)=C$ and 
$$ \tau^*\left(\frac{x-y}{y-1}\right)=\frac{x-y}{y-1}.  $$ 
Therefore, the point $(1:1:1)$ is quasi-Galois on $\mathbb{P}^2 \setminus \{XYZ=0\}$. 
By considering the actions $\sigma_1$ and $\sigma_2$, we have $\delta'[3] \ge 12$. 
By Theorem \ref{outer}(1), we have $\delta'[3]=12$. 
\end{proof} 

On the automorphism group, we have the following. 

\begin{theorem} \label{sextic, generated}
Let $C$ be the plane curve defined by $X^6+Y^6+Z^6-10(X^3Y^3+Y^3Z^3+Z^3X^3)=0$. 
Then, 
$$ G_3(C)={\rm Aut}(C). $$
\end{theorem}

\begin{proof}
Let $P_1=(1:0:0)$, $P_2=(0:1:0)$, $P_3=(0:0:1)$, and let $P_{ij}=(\zeta^i:\zeta^j:1)$ for $i, j=0, 1, 2$, where $\zeta$ is a primitive cubic root of unity. 
Then, the set $\Delta':=\{P_1, P_2, P_3\} \cup \{P_{ij} \ | \ i, j=0, 1, 2\}$ consists of all quasi-Galois points $P$ with $|G[P]|=3$. 

Let $\sigma \in {\rm Aut}(C) \subset {\rm Aut}(\mathbb{P}^2)$.
Then, $\sigma$ acts on $\Delta'$. 
If $\sigma(P_1)=P_2$ or $P_3$, then there exists $\phi_1 \in G_3(C)$ such that $\phi_1\sigma(P_1)=P_1$,  since the automorphisms $(X:Y:Z) \mapsto (Y:X:Z)$ and $(X:Y:Z) \mapsto (Z:Y:X)$ are contained in $G_3(C)$ as in the proof of Theorem \ref{sextic}. 
If $\sigma(P_1)=P_{ij}$ for some $i, j$, then there exists $\phi_2 \in G[P_{kj}]$ for $k \ne i$ such that $\phi_2\sigma(P_1)=P_1$. 
Therefore, there exists $\phi \in G_3(C)$ such that $\phi\sigma(P_1)=P_1$. 
The line $F[P_1]\setminus \{P_1\}$ is a unique line $\ell$ such that $I_Q(C, \overline{P_1Q})=3$ for any $Q \in C \cap \ell$. 
By this fact and $\phi\sigma(P_1)=P_1$, $\phi\sigma(F[P_1]\setminus\{P_1\})=F[P_1]\setminus \{P_1\}$. 
Since $\Delta' \cap F[P_1]\setminus \{P_1\}=\{P_2, P_3\}$ and the automorphism $(X:Y:Z) \mapsto (X:Z:Y)$ is contained in $G_3(C)$, there exists $\phi_3 \in G_3(C)$ such that $\phi_3\sigma(P_i)=P_i$ for $i=1, 2, 3$. 
Then, $\phi_3\sigma$ is represented by the matrix of the form 
$$ \left(\begin{array}{ccc}
\alpha & 0 & 0 \\
0 & \beta & 0 \\
0 & 0 & 1 
\end{array}\right)$$
for some $\alpha, \beta \in K$. 
By considering the action on the defining equation, we have $\alpha^3=1$ and $\beta^3=1$. 
If we take 
$$
\phi_4=\left(\begin{array}{ccc}
\alpha^{-1} & 0 & 0 \\
0 & 1 & 0 \\
0 & 0 & 1 
\end{array}\right) \in G[P_1], \ 
\phi_5=\left(\begin{array}{ccc}
1 & 0 & 0 \\
0 & \beta^{-1} & 0 \\
0 & 0 & 1 
\end{array}\right) \in G[P_2],$$ 
then $\phi_5\phi_4\phi_3\sigma=1$ on $\mathbb{P}^2$. 
Therefore, $\sigma=\phi_3^{-1}\phi_4^{-1}\phi_5^{-1} \in G_3(C)$. 
\end{proof}

\begin{remark}
For the curve defined by $X^6+Y^6+Z^6-10(X^3Y^3+Y^3Z^3+Z^3X^3)=0$, it is known that the group ${\rm Aut}(C)$ is isomorphic to the Hessian group of order $216$ (\cite{artebani-dolgachev}). 
\end{remark}

\section{Fermat curves} 
In this section, we assume that $C$ is the Fermat curve $X^d+Y^d+Z^d=0$ of degree $d \ge 4$. 
Let $P_1=(1:0:0)$, $P_2=(0:1:0)$, and let $P_3=(0:0:1)$. 
It follows from Yoshihara--Miura's theorem \cite{miura-yoshihara1, yoshihara1} that $|G[P_i]|=d$ (i.e. $P_i$ is outer Galois) for each $i$, and $\delta'[d]=3$.

\begin{proposition} \label{Fermat1}
\begin{itemize}
\item[(1)] Let $d=2n$, and let $\zeta$ satisfy $\zeta^{2n}=1$. 
Then, points $(\zeta:0:1)$, $(1:\zeta:0)$ and $(1:0:\zeta)$ are quasi-Galois points in $\mathbb P^2 \setminus C$. 
\item[(2)] Let $d=2n+1$, and let $\eta$ satisfy $\eta^{2n+1}=-1$. 
Then, points $(\eta:0:1)$, $(1:\eta:0)$ and $(1:0:\eta)$ are quasi-Galois points in $C$. 
\end{itemize}
\end{proposition}

\begin{proof} 
Let $P$ be one of the points given in (1). 
Since $C$ is invariant under the linear transformations $(X:Y:Z) \mapsto (Z:X:Y)$ and $(X:Y:Z) \mapsto (\zeta^{-1}X:Y:Z)$, we can assume that $P=(1:0:1)$. 
We set 
$$ \tilde{X}=\frac{1}{2}(X+Z),\ \tilde{Y}=Y,\ \tilde{Z}=\frac{1}{2}(X-Z)$$
and take the linear transformation $\phi:(X:Y:Z) \mapsto (\tilde{X}:\tilde{Y}:\tilde{Z})$.
Then, $\phi^{-1}(P)=(1:0:0)$ and $\phi^{-1}(C)$ is given by 
$$ (X+Z)^{2n}+(X-Z)^{2n}+Y^{2n}=2\sum_{k=0}^{n} \binom{2n}{2k}X^{2n-2k}Z^{2k}+Y^{2n}= 0. $$
By Theorem \ref{standard form}, $\phi^{-1}(P)$ is quasi-Galois. 
Therefore, $P$ is quasi-Galois. 

Let $P$ be one of the points given in (2). 
Since $C$ is invariant under the linear transformations $(X:Y:Z) \mapsto (Z:X:Y)$ and $(X:Y:Z) \mapsto (-\eta^{-1}X:Y:Z)$, we can assume that $P=(-1:0:1)$. 
We set 
$$ \tilde{X}=\frac{1}{2}(X+Z),\ \tilde{Y}=Y,\ \tilde{Z}=\frac{1}{2}(X-Z)$$
and take the linear transformation $\phi:(X:Y:Z) \mapsto (\tilde{X}:\tilde{Y}:\tilde{Z})$.
Then, $\phi^{-1}(P)=(0:0:1)$ and $\phi^{-1}(C)$ is given by 
$$ (X+Z)^{2n+1}+(X-Z)^{2n+1}+Y^{2n+1}=2\sum_{k=0}^{n} \binom{2n}{2k}X^{2n+1-2k}Z^{2k}+Y^{2n+1}= 0. $$
By Theorem \ref{standard form}, $\phi^{-1}(P)$ is quasi-Galois. 
Therefore, $P$ is quasi-Galois. 
\end{proof}

\begin{proposition} \label{Fermat2}
If a point $P$ satisfies $1<|G[P]|<d$, then $|G[P]|=2$ and $P \in \{XYZ=0\}$. 
Furthermore, the number of such points is at most $3d$. 
\end{proposition}

\begin{proof}
We prove that $P \in \{ XYZ=0 \}$. 
Assume by contradiction that $P \not \in \{ XYZ=0 \}$. 
Let $\sigma \in G[P] \setminus \{1\}$. 
There exists a unique outer Galois point on the line $\overline{PP_1}$. 
Then, $\sigma(P_1)=P_1$. 
Similarly, $\sigma(P_2)=P_2$ and $\sigma(P_3)=P_3$. 
By Corollary \ref{fixed locus}, $\sigma(P)=P$. 
Since $\sigma$ fixes non-collinear three points $P_1, P_2, P_3$ and the point $P \not\in \bigcup_{i \ne j}\overline{P_iP_j}$, $\sigma$ is identity on $\mathbb P^2$. 
This is a contradiction. 
 
We can assume that $P \in \overline{P_1P_2}$. 
Let $\sigma \in G[P]$ be a generator. 
Since outer Galois points on the line $\overline{P_1P_2}$ are exactly two, $\sigma$ acts on the set $\{P_1, P_2\}$. 
Note that $\sigma(P)=P$. 
Then, the restriction $\sigma|_{\overline{P_1P_2}}$ is of order two. 
By Theorem \ref{standard form}, the order of $\sigma$ is equal to the order of $\sigma|_\ell$ for each line $\ell \ni P$. 
Therefore, $|G[P]|=2$. 

We prove the latter assertion. 
If $d$ is odd, then $P \in C$. 
The assertion is obvious, since the number of points in $C \cap \{XYZ=0\}$ is $3d$. 
Assume that $d$ is even. 
Let $P \in \overline{P_1P_2}$ with $|G[P]|=2$, and let $\sigma \in G[P]$ be a generator. 
Let $\zeta$ be a primitive $d$-th root of unity, and let $\eta$ be a $d$-th root of $-1$. 
Then, $C \cap \overline{P_1P_2}=\{(\zeta^i\eta:1:0) \ |\ i=0, \ldots, d-1\}$. 
We identify $\overline{P_1P_2}$ with $\mathbb P^1$ by the isomorphism $(X:Y:0) \mapsto (X:Y)$. 
Then, $P_1=(1:0)$ and $P_2=(0:1)$. 
Since $\sigma$ acts on the set $\{P_1, P_2\}$, $\sigma$ is represented by the matrix 
$$ \left(\begin{array}{cc} 0 & a \\ 1 & 0 \end{array} \right). $$
Then, 
$$ \left(\begin{array}{cc} 0 & a \\ 1 & 0 \end{array} \right)\left(\begin{array}{c} \zeta^i \eta \\ 1 \end{array} \right)=\left(\begin{array}{c} a \\ \zeta^i \eta \end{array}\right).  $$
Since $\sigma$ acts on $C \cap \overline{P_1P_2}$, there exists $j$ such that $(a:\zeta^i\eta)=(\zeta^j\eta:1)$. 
We have $a=\zeta^{k}\eta^2$ for some $k$. 
We consider the fixed locus of $\sigma$. 
Assume that 
$$ \left(\begin{array}{c} x \\ 1 \end{array} \right)=\left(\begin{array}{cc} 0 & \zeta^k\eta^2 \\ 1 & 0 \end{array} \right)\left(\begin{array}{c} x \\ 1 \end{array} \right)=\left(\begin{array}{c} \zeta^k\eta^2 \\ x \end{array}\right).$$
Then, $x^2=\zeta^k\eta^2$ and $x^d=(\zeta^{\frac{d}{2}})^k\eta^d=(-1)^k(-1)$. 
If $x^d=-1$, then $(x:1) \in C \cap \overline{P_1P_2}$. 
If we take $P=(x:1) \not\in C$, then $x^d=1$. 
Therefore, we have at most $d$ quasi-Galois points on the line $\overline{P_1P_2}$. 
\end{proof}

\begin{theorem} \label{Fermat, number}
Let $C$ be the Fermat curve of degree $d$. 
Then, we have the following.
\begin{itemize}
\item[(1)] $\delta'[d]=3$. 
\item[(2)] If $d$ is even, then $\delta'[\ge 2]=\delta'[d]+\delta'[2]=3+3d$ and $\delta[\ge 2]=0$. 
\item[(3)] If $d$ is odd, then $\delta'[\ge 2]=\delta'[d]=3$ and $\delta[\ge 2]=\delta[2]=3d$. 
\end{itemize} 
\end{theorem}

\begin{proof}
Assertion (1) is nothing but Yoshihara--Miura's theorem \cite{miura-yoshihara1, yoshihara1}. 
By Propositions \ref{Fermat1} and \ref{Fermat2}, we have assertions (2) and (3).  
\end{proof}

\begin{theorem} \label{Fermat, generated}
$$ \langle G_2(C), G_d(C) \rangle={\rm Aut}(C).  $$
\end{theorem}

\begin{proof}
Let $\sigma \in {\rm Aut}(C) \subset {\rm Aut}(\mathbb P^2)$. 
Then, $\sigma$ acts on the set $\{P_1, P_2, P_3\}$ of outer Galois points. 
If $\sigma(P_1)=P_2$, then there exists $\phi_1 \in G_2(C)$ such that $\phi_1\sigma(P_1)=P_1$, by using an action associated with a quasi-Galois point on the line $\overline{P_1P_2}$.  
If $\phi_1\sigma(P_2)=P_3$, then we take $\phi_2 \in G_2(C)$ such that $\phi_2(P_1)=P_1$ and $\phi_2(P_3)=P_2$, which comes from an action associated with a quasi-Galois point on the line $\overline{P_2P_3}$. 
Therefore, there exists $\phi \in G_2(C)$ such that $\phi\sigma(P_i)=P_i$ for $i=1, 2, 3$. 
Then, $\phi\sigma$ is represented by the matrix of the form
$$ \left(\begin{array}{ccc} \alpha & 0 & 0 \\ 0 & \beta & 0 \\ 0 & 0 & 1 \end{array} \right) $$
for some $\alpha, \beta \in K$. 
By considering the action on the defining equation $X^d+Y^d+Z^d=0$, we have $\alpha^d=1$ and $\beta^d=1$. 
If we take 
$$ \phi_3=\left(\begin{array}{ccc} \alpha^{-1} & 0 & 0 \\ 0 & 1 & 0 \\ 0 & 0 & 1 \end{array} \right) \in G[P_1], \ \phi_4=\left(\begin{array}{ccc} 1 & 0 & 0 \\ 0 & \beta^{-1} & 0 \\ 0 & 0 & 1 \end{array} \right) \in G[P_2],$$
then $\phi_4\phi_3\phi\sigma=1$ on $\mathbb P^2$. 
Therefore, $\sigma=\phi^{-1}\phi_3^{-1}\phi_4^{-1} \in \langle G_2(C), G_d(C) \rangle$. 
\end{proof} 

\begin{remark}
According to \cite[Example 2]{miura-ohbuchi}, $G_d(C) \ne {\rm Aut}(C)$. 
\end{remark}

\section{The number of quasi-Galois points for smooth quartic curves} 
In this section, we assume that $C$ is smooth and of degree $d=4$. 
The set of all quasi-Galois points in $\mathbb P^2 \setminus C$ for $C$ is denoted by $\Delta'$.  
If $P \in \Delta'$, then there exists a unique involution in $G[P]$, since $G[P]$ is a cyclic group of order $2$ or $4$. 
First, we note the following. 

\begin{lemma} \label{four cover}
If $P \in \Delta'$, then we have the following. 
\begin{itemize}
\item[(1)] There exist exactly four lines $\ell \ni P$ such that $C \cap \ell$ consists of one or two points, and the tangent line at each point of $C \cap \ell$ is equal to $\ell$. 
\item[(2)] There does not exist a line $\ell \ni P$ such that $I_Q(C, \ell)=3$ for some $Q \in C \cap \ell$.  
\end{itemize}
\end{lemma} 

\begin{proof}
Let $\sigma \in G[P]$ be the involution. 
The projection $\pi_P$ is the composite map of $g_P: C \rightarrow C/\sigma$ and $f_P:C/\sigma \rightarrow \mathbb P^1$. 
Since $g_P$ is ramified at exactly four points by Corollary \ref{fixed locus} and Fact \ref{Galois covering}(1), by Hurwitz formula, the genus of the smooth model of $C/\sigma$ is equal to $1$. 
Then, $f_P: C/\sigma \rightarrow \mathbb P^1$ has exactly four ramification points. 
Therefore, we have (1). 
Assertion (2) is obvious, since $\pi_P$ is the composite map of double coverings $g_P$ and $f_P$. 
\end{proof}

We recall the notion of $G$-pairs and the following proposition. 

\begin{proposition} \label{pair 2}
Let $(P_1, P_2)$ be a $G$-pair. 
Then, there exists a linear transformation $\phi$ such that $\phi(P_1)=(1:0:0)$, $\phi(P_2)=(0:1:0)$, and $\phi(C)$ is given by 
$$ X^4+Y^4+Z^4+aX^2Y^2+bY^2Z^2+cZ^2X^2=0, $$
where $a, b, c \in K$. 
In particular, $C \cap \overline{P_1P_2}$ consists of exactly four points. 
\end{proposition} 

\begin{proof} 
By the assumption, $P_2 \in F[P_1]$ and $P_1 \in F[P_2]$. 
By Theorem \ref{standard form}, for a suitable system of coordinates, we can assume that $P_1=(1:0:0)$ and the involution $\sigma_1 \in G[P_1]$ is represented by the matrix 
$$A_{\sigma_1}=\left(\begin{array}{ccc} -1 & 0 & 0 \\ 0 & 1 & 0 \\ 0 & 0 & 1 \end{array}\right). 
$$ 
Then, the line given by $F[P_1]\setminus \{P_1\}$ is defined by $X=0$. 
Since $P_2 \in F[P_1] \setminus \{P_1\}$, $P_2=(0:1:a)$ for some $a \in K$. 
If we take the linear transformation $(X:Y:Z) \mapsto (X:Y:Z-aY)$, we can assume that $P_2=(0:1:0)$. 
Then, the involution $\sigma_2 \in G[P_2]$ is represented by the matrix 
$$A_{\sigma_2}=\left(\begin{array}{ccc} 1 & 0 & 0 \\ a & -1 & b \\ 0 & 0 & 1 \end{array}\right), $$ 
for some $a,b \in K$.  
Since the line given by $F[P_2]\setminus \{P_2\}$ is defined by $aX+(-1-1)Y+bZ=0$ and $P_1 \in F[P_2] \setminus \{P_2\}$, we have $a=0$. 
If we take 
$$ B=\left(\begin{array}{ccc} 2 & 0 & 0 \\ 0 & 1 & b \\ 0 & 0 & 2 \end{array}\right), $$
then 
$$ B^{-1}A_{\sigma_1}B=\left(\begin{array}{ccc} -1 & 0 & 0 \\ 0 & 1 & 0 \\ 0 & 0 & 1 \end{array}\right), \ B^{-1}A_{\sigma_2}B=\left(\begin{array}{ccc} 1 & 0 & 0 \\ 0 & -1 & 0 \\ 0 & 0 & 1 \end{array}\right).  $$
By taking the linear transformation represented by $B^{-1}$, we have the defining polynomial
$$ F=X^{4}+(aY^2+bZ^2)X^2+(\alpha Y^{4}+\beta Y^2Z^2+\gamma Z^{4}) $$
of $C$. 
We can assume $\alpha=\gamma=1$. 
We have the conclusion. 
\end{proof} 

Let $\ell \subset \mathbb P^2$ be a projective line. 
We would like to calculate the number of quasi-Galois points on the line $\ell$.  
We treat the cases $\# C \cap \ell=4, 3, 2$ and $1$ separately. 

\begin{proposition} \label{four points} 
Let $\ell$ be a line with $\# C \cap \ell=4$. 
Then, $\#\Delta' \cap \ell=0, 1, 2, 4$ or $6$.
Furthermore, if $\#\Delta' \cap \ell=2$ (resp. $4$, $6$), then we have exactly one (resp. two, three) $G$-pair. 
\end{proposition} 

\begin{proof}
Let $C \cap \ell=\{Q_1, Q_2, Q_3, Q_4\}$. 
We consider the possibilities of involutions acting on $C \cap \ell$. 
There are at most three types: 
\begin{itemize}
\item[(1)] $Q_1 \leftrightarrow Q_2$,\ $Q_3 \leftrightarrow Q_4$, 
\item[(2)] $Q_1 \leftrightarrow Q_3$,\ $Q_2 \leftrightarrow Q_4$, 
\item[(3)] $Q_1 \leftrightarrow Q_4$,\ $Q_2 \leftrightarrow Q_3$. 
\end{itemize}
If $P_1, P_2 \in \Delta' \cap \ell$, and involutions $\sigma_1 \in G[P_1]$ and $\sigma_2 \in G[P_2]$ are of type (1), then we have $\sigma_1|_\ell=\sigma_2|_\ell$. 
Then, $\sigma_1(P_2)=\sigma_2(P_2)=P_2$ and $\sigma_2(P_1)=\sigma_1(P_1)=P_1$, i.e. $(P_1, P_2)$ is a $G$-pair. 
For each types (1)-(3) we have at most two quasi-Galois points, and hence, $\#\Delta' \cap \ell \le 6$.

Let $\sigma_1 \in G[P_1]$ and $\sigma_2 \in G[P_2]$ give involutions of types (1) and (2) respectively. 
Then, $\sigma_1\sigma_2\sigma_1(Q_1)=Q_3$, and hence, $\sigma_1\sigma_2\sigma_1$ is of type (2). 
Since $\sigma_1(P_2) \ne P_2$ and $\sigma_1(P_2)$ is quasi-Galois, $(P_2, \sigma_1(P_2))$ is a $G$-pair. 
Similarly, $(P_1, \sigma_2(P_1))$ is a $G$-pair. 
We have two $G$-pairs. 

Assume that $\#\Delta'\cap\ell \ge 5$. 
There are at least two $G$-pairs. 
We can assume that $(P_1, P_2)$ and $(P_3, P_4)$ are $G$-pairs, and give involutions on $\ell$ of type (1) and (2) respectively. 
Let $P_5$ be another quasi-Galois point, and let $\sigma_i \in G[P_i]$ be the involution. 
Then, $\sigma_1(P_5) \ne P_5$ and the involution $\sigma_1\sigma_5\sigma_1 \in G[\sigma_1(P_5)]$ gives an involution on $\ell$ of type (3). 
Therefore, $\sigma_1(P_5) \ne P_1, \ldots, P_5$. 
We have $\#\Delta'\cap\ell=6$.
\end{proof}

If $(P_1, P_2)$ is a $G$-pair, then, up to linear transformations, $P_1=(1:0:0)$, $P_2=(0:1:0)$,  and $C$ is defined by 
$$ F=X^4+Y^4+Z^4+aX^2Y^2+bY^2Z^2+cZ^2X^2=0. $$
We consider this curve. 
Then, the lines $F[P_1]\setminus \{P_1\}$ and $F[P_2]\setminus \{P_2\}$ are defined by $X=0$ and $Y=0$ respectively. 
Since $C$ is smooth, we have $a \ne \pm 2$, $b \ne \pm 2$ and $c \ne \pm 2$.

\begin{proposition} \label{four-six points 1}
We have the following. 
\begin{itemize}
\item[(1)] If there exist two $G$-pairs on the line defined by $Z=0$, then $b=\pm c$. 
Furthermore, when $c=-b$, we take the linear transformation given by $X \mapsto iX$, where $i^2=-1$, so that we have the defining equation with $c=b$. 
\item[(2)] If there exist three $G$-pairs on the line defined by $Z=0$, then $b=c=0$. 
\end{itemize} 
\end{proposition} 

\begin{proof} 
Assume that $(P_1, P_2)$ and $(P_3, P_4)$ are two $G$-pairs on the line $Z=0$. 
Let $Q \in F[P_1] \cap F[P_2]$. 
Then $Q=(0:1:0)$. 
Let $\sigma_1\in G[P_1]$ and $\sigma_3 \in G[P_3]$ be involutions. 
Then, $\sigma_1\sigma_3$ satisfies 
$$ P_1 \leftrightarrow P_2,\ P_3 \leftrightarrow P_4 $$
(see the second paragraph of the proof of Proposition \ref{four points}). 
Since $\sigma_1\sigma_3(F[P_1])=F[P_2]$, we have $\sigma_1\sigma_3(Q)=Q$. 
Then, $\sigma_1\sigma_3$ is represented by the matrix 
$$ \left(\begin{array}{ccc} 0 & \lambda & 0 \\ 
\mu & 0 & 0 \\
0 & 0 & 1 
\end{array}\right) $$
for some $\lambda, \mu \in K$. 
Then, $((\sigma_1\sigma_3)^{-1})^*F$ and $F$ are the same up to a constant. 
Therefore, we have  
$$ \lambda^4Y^4+\mu^4X^4+Z^4+a\lambda^2\mu^2X^2Y^2+b\mu^2X^2Z^2+c\lambda^2Y^2Z^2=F. $$
Considering the coefficients of $Y^4$ and $Y^2Z^2$, we have $\lambda^2=\pm 1$ and $b=\pm c$. 

Assume that $(P_1, P_2)$, $(P_3, P_4)$ and $(P_5, P_6)$ are three $G$-pairs on the line $Z=0$. 
Let $\sigma_3\in G[P_3]$ and $\sigma_5 \in G[P_5]$ be involutions. 
Then, $\sigma_3\sigma_5$ satisfies 
$$ P_1 \rightarrow P_1,\ P_2 \rightarrow P_2,\ P_3 \leftrightarrow P_4,\ P_5 \leftrightarrow P_6 $$
(see the second paragraph of the proof of Proposition \ref{four points}). 
Since $\sigma_3\sigma_5(P_i)=P_i$ for $i=1, 2$, we have $\sigma_3\sigma_5(Q)=Q$. 
Note that the order of $\sigma_3\sigma_5$ is at least $3$. 
Then, $\sigma_3\sigma_5$ is represented by the matrix 
$$ \left(\begin{array}{ccc} -\eta & 0 & 0 \\ 
0 & \eta & 0 \\
0 & 0 & 1 
\end{array}\right),$$
where $\eta^2 \ne -1$. 
Then, $((\sigma_3\sigma_5)^{-1})^*F$ and $F$ are the same up to a constant. 
Therefore, we have  
$$ \eta^4X^4+\eta^4Y^4+Z^4+a\eta^2X^2Y^2+b\eta^2Y^2Z^2+c\eta^2Z^2X^2=F. $$
Considering the coefficients of $Y^2Z^2$ and $Z^2X^2$, we have $b=c=0$. 
\end{proof}

On the contrary, we have the following. 

\begin{proposition} \label{four-six points 2} 
Let $a, b \in K$, and let $C$ be the smooth plane curve given by 
$$ X^4+Y^4+Z^4+aX^2Y^2+bY^2Z^2+bZ^2X^2=0. $$
Then, we have the following. 
\begin{itemize}
\item[(1)] Points $(1:0:0)$, $(0:1:0)$, $(1:1:0)$ and $(1:-1:0)$ are quasi-Galois points. Furthermore, if $b \ne 0$, they are not Galois. 
\item[(2)] If $b=0$, then points $(\pm i:1:0)$ are quasi-Galois, where $i^2=-1$. 
Furthermore, we have the following. 
\begin{itemize} 
\item{} If $a \ne 0$, then points $(1:0:0)$ and $(0:1:0)$ are not Galois. 
\item{} If $a \ne 6$, then points $(\pm 1:1:0)$ are not Galois.
\item{} If $a \ne -6$, then points $(\pm i:1:0)$ are not Galois. 
\item{} If $a =0$ or $\pm 6$, then there exists a linear transformation $\phi$ such that $\phi(\{Z=0\})=\{Z=0\}$ and $\phi(C)$ is the Fermat curve $X^4+Y^4+Z^4=0$. 
\end{itemize} 
\item[(3)] If $b=0$, then $\delta'[2]=6$ or $12$. 
Furthermore, $\delta'[2]=12$ if and only if $C$ is projectively equivalent to the Fermat curve $X^4+Y^4+Z^4=0$. 
\end{itemize} 
\end{proposition}

\begin{proof}
We consider points $(\pm 1:1:0)$.  
We set 
$$ \tilde{X}=\frac{1}{2}(X+Y),\ \tilde{Y}=\frac{1}{2}(X-Y),\ \tilde{Z}=Z$$
and take the linear transformation $\phi:(X:Y:Z) \mapsto (\tilde{X}:\tilde{Y}:\tilde{Z})$.
Then, $\phi^{-1}((1:1:0))=(1:0:0)$, $\phi^{-1}((-1:1:0))=(0:1:0)$, and $\phi^{-1}(C)$ is given by 
$$ G=(2+a)X^4+(2+a)Y^4+Z^4+(12-2a)X^2Y^2+2bY^2Z^2+2bX^2Z^2=0. $$
By Theorem \ref{standard form}, $\phi^{-1}((\pm 1:1:0))$ are quasi-Galois. 
Therefore, $(\pm1:0:0)$ are quasi-Galois. 
Furthermore, if $\phi^{-1}((1:1:0))$ is Galois, then the matrix
$$ \left(\begin{array}{ccc} i & 0 & 0 \\
0 & 1 & 0 \\
0 & 0 & 1 \end{array}\right)$$
acts on $G$. 
This implies $12-2a=0$ and $2b=0$. 

Let $b=0$. 
We consider points $(\pm i:1:0)$.  
We set 
$$ \tilde{X}=\frac{1}{2}(X+iY),\ \tilde{Y}=\frac{1}{2}(X-iY),\ \tilde{Z}=Z$$
and take the linear transformation $\phi:(X:Y:Z) \mapsto (\tilde{X}:\tilde{Y}:\tilde{Z})$.
Then, $\phi^{-1}((i:1:0))=(1:0:0)$, $\phi^{-1}((-i:1:0))=(0:1:0)$, and $\phi^{-1}(C)$ is given by 
$$ H=(2-a)X^4+(2-a)Y^4+Z^4+(12+2a)X^2Y^2=0. $$
By Theorem \ref{standard form}, $\phi^{-1}((\pm i:1:0))$ are quasi-Galois. 
Therefore, $(\pm i:0:0)$ are quasi-Galois. 
Furthermore, if $a \ne -6$, then $(\pm i:0:0)$ are not Galois. 

We prove (3). 
Now, we have six quasi-Galois points on the line $Z=0$. 
By the defining equation, we infer that $Q=(0:0:1)$ is an outer Galois point and the set $F[Q]\setminus \{Q\}$ is given by $Z=0$. 
Assume that $\delta'[2]>6$. 
Then there exists a quasi-Galois point $R \in \mathbb P^2 \setminus (\{Z=0\} \cup \{Q\})$. 
Let $\tau \in G[R]$ be the involution. 
If $\tau(Q)=Q$, then by Lemma \ref{pair 1}, $(R, Q)$ is a $G$-pair. 
Then $R$ must lie on the line $Z=0$. 
This is a contradiction. 
If $\tau(Q) \ne Q$ is not a $G$-pair, then we have two Galois points. 
It follows from a theorem of Yoshihara \cite{yoshihara1} that $C$ is projectively equivalent to the Fermat curve. 
By Theorem \ref{Fermat, number}, we have $\delta'[2]=12$.  
\end{proof} 

We consider the case where $C \cap \ell$ consists of three points. 

\begin{proposition} \label{three points}
If $\# C \cap \ell=3$, then $\#\Delta'\cap\ell=0$ or $1$. 
\end{proposition}

\begin{proof}
Let $C \cap \ell=\{Q_1, Q_2, Q_3\}$, let $T_{Q_1}C=\ell$, and let $P_1, P_2 \in \ell$ be different quasi-Galois points. 
Then, $Q_1 \in C \cap F[P_1] \cap F[P_2]$. 
This is a contradiction to Proposition \ref{two quasi-Galois}. 
\end{proof}

We consider the case where $C \cap \ell$ consists of two points.

\begin{proposition} \label{two points} 
Let $C \cap \ell=\{Q_1, Q_2\}$, where $Q_1 \ne Q_2$. 
\begin{itemize} 
\item[(1)] If $I_{Q_1}(C, \ell)=3$, then $\#\Delta'\cap\ell=0$. 
\item[(2)] If $T_{Q_1}C=T_{Q_2}C=\ell$, then $\#\Delta'\cap\ell=0, 1$ or $3$.
\end{itemize}
\end{proposition} 

\begin{proof}
Assertion (1) is derived from Lemma \ref{four cover}(2). 
We consider assertion (2). 
Let $P_1, P_2 \in \ell$ be quasi-Galois points, and let $\sigma_i \in G[P_i]$ be the involution. 
By Proposition \ref{pair 2}, $(P_1, P_2)$ is not a $G$-pair. 
By Lemma \ref{pair 1}, $\sigma_1(P_2) \ne P_2$, and hence, $\#\Delta'\cap\ell\ge 3$. 
We consider $\sigma_1\sigma_2$. 
Then, $\sigma_1\sigma_2(Q_1)=Q_1$ and $\sigma_1\sigma_2(Q_2)=Q_2$. 
Let $R$ be the intersection point of the lines $F[P_1]\setminus \{P_1\}$ and $F[P_2]\setminus \{P_2\}$. 
Then, $\sigma_1\sigma_2(R)=\sigma_1(R)=R$.  
If $R \in C$, then $T_RC \ni P_1, P_2$. 
Therefore, $R \not\in C$. 
For a suitable system of coordinates, we can assume that $Q_1=(1:0:0), Q_2=(0:1:0)$ and $R=(0:0:1)$. 
Consider the action of $\sigma_1\sigma_2$ on the lines $\overline{Q_1R}$ and $\overline{Q_2R}$. 
Since $\sigma_1\sigma_2$ fixes $Q_1, Q_2$ and $R$, $\sigma_1\sigma_2|_{\overline{Q_iR}}$ is identity if $\sigma_1\sigma_2$ fixes some point of $C \cap \overline{Q_iR}$ other than $Q_i$.
Therefore, the restriction $\sigma_1\sigma_2|_{\overline{Q_iR}}$ is identity or of order three for $i=1, 2$. 
Then, $\sigma_1\sigma_2$ is represented by the matrix 
$$ A_{\sigma_1\sigma_2}=\left(\begin{array}{ccc} \zeta & 0 & 0 \\ 0 & \eta & 0 \\ 0 & 0 &1 \end{array} \right), $$
where $\zeta$ and $\eta$ are cubic roots of $1$. 
Since $Q_1$ and $Q_2$ are not inner Galois (by Facts \ref{index} and \ref{Galois covering}(2)) and $R \not\in C$, we have $\eta \ne 1$, $\zeta \ne 1$ and $\zeta \ne \eta$. 
This implies that $\eta=\zeta^2$.

Let $P_3:=\sigma_1\sigma_2(P_2)$ ($=\sigma_1(P_2)$). 
Since $\sigma_1\sigma_2(F[P_2])=F[P_3]$, $R \in F[P_3]$. 
Assume by contradiction that $\#\Delta'\cap\ell \ge 4$. 
Let $P_4 \ne P_1, P_2, P_3$ be quasi-Galois, and let $\sigma_4 \in G[P_4]$ be the involution. 
Since $\sigma_1\sigma_4(Q_i)=Q_i$ for $i=1, 2$ and the order of $\sigma_1\sigma_4$ is three, we have $\sigma_1\sigma_4|_{\ell}=\sigma_1\sigma_2|_{\ell}$ or $(\sigma_1\sigma_2)^2|_{\ell}$. 
Then, we find that $\sigma_1\sigma_4(R)=R$, and hence, $\sigma_1\sigma_4=\sigma_1\sigma_2$ or $(\sigma_1\sigma_2)^2$ on $\mathbb P^2$. 
We have $\sigma_4=\sigma_2 \in G[P_2]$ or $\sigma_4=\sigma_2\sigma_1\sigma_2 \in G[P_3]$. 
This is a contradiction. 
\end{proof}

We consider the case where $C \cap \ell$ consists of a unique point.

\begin{proposition} \label{one point} 
If $\# C \cap \ell=1$, then $\#\Delta'\cap\ell=0$ or $1$. 
\end{proposition} 

\begin{proof}
Let $C \cap \ell=\{Q\}$, and let $P_1, P_2$ be different quasi-Galois points. 
Then, $Q \in F[P_1] \cap F[P_2]$. 
This is a contradiction to Proposition \ref{two quasi-Galois}. 
\end{proof}

\begin{theorem} \label{quartic} 
Let $C \subset \mathbb P^2$ be a smooth curve of degree four. 
Then, 
$$ \delta'[2] \le 21. $$
Furthermore, the equality holds if and only if $C$ is projectively equivalent to the curve defined by 
$$ X^4+Y^4+Z^4+a(X^2Y^2+Y^2Z^2+Z^2X^2)=0, $$
where $a \in K$ satisfies $a^2+3a+18=0$. 
\end{theorem} 

\begin{proof} 
Assume that $\delta'[2] \ge 13$. 
It follows from Propositions \ref{four points}, \ref{four-six points 1}, \ref{four-six points 2}, \ref{three points}, \ref{two points} and \ref{one point} that $\#\Delta'\cap\ell \le 4$ for each line $\ell$, and $\#\Delta'\cap\ell\ge 2$ only if $\# C \cap \ell=4$ or $\ell$ is the tangent line at two distinct points of $C \cap \ell$. 
Let $P \in \Delta'$.
By Lemma \ref{four cover} and Proposition \ref{two points}, there exist at most four lines $\ell \ni P$ such that $\# C \cap \ell \le 2$ and $\#\Delta' \cap \ell \ge 2$ (in this case $\#\Delta' \cap \ell=3$). 
By Proposition \ref{four points}, if a line $\ell \ni P$ satisfies $\# C \cap \ell=4$ and $\#\Delta'\cap\ell\ge 2$, then there exists a point $P' \in \Delta' \cap \ell$ such that $(P, P')$ is a $G$-pair. 
Then, $P' \in F[P]\setminus \{P\}$. 
Since the set $F[P]\setminus \{P\}$ is a line, $\# \Delta' \cap F[P]\setminus\{P\} \le 4$. 
Therefore, we have at most four lines $\ell \ni P$ such that $\# C \cap \ell=4$ and $\#\Delta' \cap \ell \ge 2$. 
As a consequence, we have an inequality $\delta'[2] \le 1+2 \times 4+3\times 4=21$. 

Assume that $\delta'[2]=21$.  
Let $P=(1:0:0) \in \Delta'$. 
By the above discussion, there exist four lines $\ell_1, \ell_2, \ell_3, \ell_4 \ni P$ such that $\# C \cap \ell_i=4$ and $\#\Delta' \cap \ell_i=4$ for $i=1, 2, 3, 4$. 
By Proposition \ref{four points}, for each $i$, there exists a point $P_i \in \ell_i$ such that $(P, P_i)$ is a $G$-pair. 
Then, by Proposition \ref{four-six points 1}, we can assume that $P=(1:0:0)$, $P_1=(0:1:0)$, $P_2=(0:0:1)$, and $C$ is given by 
$$ F=X^4+Y^4+Z^4+a(X^2Y^2+Y^2Z^2+Z^2X^2)=0 $$
for some $a \in K$. 
It follows from Proposition \ref{four-six points 2} that $P_3:=(0:1:1)$ and $P_4:=(0:-1:1) \in F[P]\setminus \{P\}=\{X=0\}$ are quasi-Galois. 
Therefore, four lines $\ell_1, \ell_2, \ell_3$ and $\ell_4$ are $\overline{PP_1}$, $\overline{PP_2}$, $\overline{PP_3}$ and $\overline{PP_4}$, which are defined by $Z=0$, $Y=0$, $Y-Z=0$ and $Y+Z=0$ respectively. 
Note that $(P_3, P_4)$ is a $G$-pair, since $(P_1, P_2)$ is a $G$-pair. 

Let $Q \in \Delta' \cap \overline{PP_4}$ other than $P$ or $P_4$, and let $\tau \in G[Q]$ be the involution. 
Note that $\tau(P_4)=P$, since $(P, P_4)$ is a $G$-pair. 
Since $(P, P_3)$ and $(P_3, P_4)$ are $G$-pairs, $F[P_3]\setminus \{P_3\}=\overline{PP_4} \ni Q$, and hence, $(P_3, Q)$ is a $G$-pair. 
Therefore $\tau(P_3)=P_3$. 
Since $\tau((0:1:1))=(0:1:1)$, $\tau((0:-1:1))=(1:0:0)$, and $\tau((1:0:0))=(0:-1:1)$, $\tau$ is represented by the matrix 
$$ \left(\begin{array}{ccc}
0 & \frac{2}{\lambda} & -\frac{2}{\lambda} \\
\lambda & 1 & 1 \\
-\lambda & 1 & 1 
\end{array}\right), $$
where $\lambda \in K$. 
Then, $(\tau^{-1})^*F$ and $F$ are the same up to a constant. 
Here 
\begin{eqnarray*} 
(\tau^{-1})^*F&=&\left(\frac{2}{\lambda}\right)^4(Y-Z)^4+(\lambda X+Y+Z)^4+(-\lambda X+Y+Z)^4\\ & &+a\left(\frac{2}{\lambda}\right)^2(Y-Z)^2(\lambda X+Y+Z)^2+a(\lambda X+Y+Z)^2(-\lambda X+Y+Z)^2 \\ & & +a\left(\frac{2}{\lambda}\right)^2(-\lambda X+Y+Z)^2(Y-Z)^2. 
\end{eqnarray*}
The coefficient of $X^2YZ$ is 
$$ 12\lambda^2+12\lambda^2-2a\left(\frac{2}{\lambda}\right)^2\lambda^2-4a\lambda^2-2a\left(\frac{2}{\lambda}\right)^2\lambda^2=0. $$
We have $\lambda^2=4a/(6-a)$. 
The coefficient of $Y^3Z$ is 
$$ -4\left(\frac{2}{\lambda}\right)^4+4+4+4a=0. $$
We have $a^3+a^2+12a-36=0$.  
Since $a \ne 2$, we have $a^2+3a+18=0$.

On the contrary, let $a^2+3a+18=0$, and let $C$ be the plane curve given by 
$$ F=X^4+Y^4+Z^4+a(X^2Y^2+Y^2Z^2+Z^2X^2)=0.  $$
By Proposition \ref{four-six points 2}, we have $9$ quasi-Galois points on the union of the lines $X=0$, $Y=0$ and $Z=0$. 
Let $\lambda$ be a solution of $\lambda^2=4a/(6-a)$, and let $\tau$ be the involution given by
$$ \left(\begin{array}{ccc}
0 & \frac{2}{\lambda} & -\frac{2}{\lambda} \\
\lambda & 1 & 1 \\
-\lambda & 1 & 1 
\end{array}\right).  
$$
Then, $\tau$ acts on $C$. 
Since $\tau(\{Y+Z=0\})=\{Y+Z=0\}$ and $\tau$ is not identity on this line, by the proof of Proposition \ref{involution}, $\tau$ is the involution of some quasi-Galois point $Q$ on the line $Y+Z=0$ other than $(1:0:0)$ or $(0:-1:1)$. 
By considering the actions associated with points $(1:0:0)$, $(0:1:0)$ and $(0:0:1)$, we have four quasi-Galois points not in $\{XYZ=0\}$. 
Note that $Q$ is different from $(\pm1:\pm1:1)$. 
Using the linear transformation given by $(X:Y:Z) \mapsto (Z:X:Y)$, we have $4 \times 3$ additional quasi-Galois points. 
Therefore, we have $\delta'[2] \ge 9+12=21$. 
Since $\delta'[2]\le 21$, we have $\delta'[2]=21$. 
\end{proof}

\begin{remark} 
It is known that the curve defined by 
$$ X^4+Y^4+Z^4+a(X^2Y^2+Y^2Z^2+Z^2X^2)=0, $$
where $a \in K$ satisfies $a^2+3a+18=0$, is projectively equivalent to the Klein quartic $X^3Y+Y^3Z+Z^3X=0$ (\cite{klein}, \cite{rodriguez-gonzalez}). 
\end{remark} 

\section{Group structure}

Let $C$ be a smooth curve. 

\begin{proposition} \label{2n^2}
If $|G[P]|=n$ and $2n=\deg \pi_P$, then $|G_P|=2n^2$. 
\end{proposition}

\begin{proof}
By Theorem \ref{standard form}, we can assume that $P=(1:0:0)$ and 
the defining equation is given by 
$$ F=X^{2n}Z+GX^n+H=0 \ (\mbox{ resp. } X^{2n}+GX^n+H=0), $$ 
if $P \in C$ (resp. $P \in \mathbb P^2 \setminus C$). 
Set $C \cap F[P]\setminus\{P\}=\{X=H=0\}=\{Q_1, \ldots, Q_d\}$. 
The projection $\pi_P$ is the composite map of $g_P: C \rightarrow C/G[P]$ and $f_P:C/G[P] \rightarrow \mathbb P^1$. 
We prove that there exists $Q_i$ such that $f_P$ is unramified at $g_P(Q_i)$. 
Assume by contradiction that for each $Q_i$, $f_P$ is ramified at $g_P(Q_i)$. 
Then, the ramification index at $Q_i$ for $\pi_P$ is equal to $\deg \pi_P$. 
Let $H_i$ be a linear homogeneous polynomial defining $\overline{PQ_i}$. 
Then, $H_i$ divides $G$ for each $i$. 
Since $\deg G < \deg H$, we have $G=0$. 
Then, by Theorem \ref{standard form}, $|G[P]|=2n$. 
This is a contradiction. 
Therefore, $f_P$ is unramified at $g_P(Q_i)$ for some $i$. 
Since $f_P$ is a double covering, there exists a point $R \in C/G[P]$ such that $f_P(R)=f_P(g_P(Q_i))$ and $f_P$ is unramified at $R$. 
Then, $\pi_P$ is unramified at each point $Q \in g_P^{-1}(R)$. 
We take a covering $h_P: \tilde{C} \rightarrow C$ such that $\pi_P \circ h_P=f_P\circ g_P \circ h_P$ gives a Galois covering. 
Since $f_P \circ g_P \circ h_P$ is ramified at every point in $(g_P\circ h_P)^{-1}(R)$ with index $\ge n$ by Fact \ref{Galois covering}(2), we have $\deg h_P \ge n$. 
Then, $|G_P| \ge 2n^2$. 
On the other hand, by Proposition \ref{Galois closure}, $|G_P|\le 2n^2$. 
\end{proof}

\begin{theorem} \label{2p^2}
Assume that $n$ is odd prime and $\deg \pi_P=2n$. 
Then, $|G[P]|=n$ if and only if $G_P \cong (\mathbb Z/n\mathbb Z) \times D_{2n}$. 
\end{theorem} 

\begin{proof}
We prove the only-if part.
By Proposition \ref{2n^2}, $|G_P|=2n^2$. 
By Sylow's theorem, there exists an element of $G_P$ of order two.  
Let $\tau$ be such an element. 
Then, we have the split exact sequence
$$ 0 \rightarrow {\rm Gal}(L_P/K(C)^{G[P]}) \rightarrow G_P \rightarrow \mathbb Z/2\mathbb Z \rightarrow 0, $$ 
by considering the orders of groups. 
Since the group of order $n^2$ is Abelian (see, for example, \cite[p. 27]{suzuki}), the Galois group ${\rm Gal}(L_P/K(C)^{G[P]})$ is $(\mathbb Z/n^2\mathbb Z)$ or $(\mathbb Z/n\mathbb Z)^{\oplus 2}$. 
Since $K(C)/K(\mathbb P^1)$ is not Galois, the group ${\rm Gal}(L_P/K(C))$ is not a normal subgroup. 
Then, $\tau^{-1}{\rm Gal}(L_P/K(C))\tau \ne {\rm Gal}(L_P/K(C))$, and hence, we have at least two subgroups of ${\rm Gal}(L_P/K(C)^{G[P]})$ of order $n$. 
Then, we have ${\rm Gal}(L_P/K(C)^{G[P]}) \cong (\mathbb Z/n\mathbb Z)^{\oplus 2}$. 
Let ${\rm Gal}(L_P/K(C))=\langle \sigma \rangle$.  
Then $\sigma(\tau^{-1}\sigma\tau)\ne 1$ is fixed by the action of $\tau$, that is, $\tau^{-1}(\sigma(\tau^{-1}\sigma\tau))\tau=\sigma(\tau^{-1}\sigma\tau)$. 
Let 
$$ A_{\tau}=\left(
\begin{array}{cc}
1 & a \\
0 & b \\
\end{array}
\right)
$$
be the matrix representing $\tau$ as an action on the vector space ${\rm Gal}(L_P/K(C)^{G[P]}) \cong (\mathbb Z/n\mathbb Z)^{\oplus 2}$ over $\mathbb Z/n\mathbb Z$. 
Since $A_{\tau}^2=1$, $b=-1$ and $A_{\tau}$ is diagonalizable. 
Therefore, we have $G_P \cong (\mathbb Z/n\mathbb Z)\times (\mathbb Z/n\mathbb Z \rtimes \mathbb Z/2\mathbb Z)$.

We prove the if part. 
Assume that $G_P \cong (\mathbb Z/n\mathbb Z)\times D_{2n}$. 
Since $|{\rm Gal}(L_P/K(C))|=n$, there exists a subgroup $H$ of order $n^2$ such that ${\rm Gal}(L_P/K(C)) \subset H$. 
We take the fixed field $L_P^H$. 
Since $L_P/L_P^H$ is Galois and $[L_P:L_P^H]=n^2$, the extension is Abelian. 
Therefore, $K(C)/L_P^H$ is Galois with $[K(C):L_P^H]=n$. 
Since $K(\mathbb P^1) \subset L_P^H$, we have $|G[P]|\ge n$. 
By the definition of $G_P$, $K(C)/K(\mathbb P^1)$ is not Galois and $|G[P]|<2n$. 
Therefore, we have $|G[P]|=n$. 
\end{proof}

\section{Relations with dual curves}
We use the same notation as in \cite{fukasawa-miura1}. 

\begin{proposition} \label{injection}
If $P$ is extendable quasi-Galois, then there exists a unique point $\overline{P} \in \mathbb P^{2*}$ such that the natural map $\sigma \mapsto \overline{\sigma}$ induces an injection
$$ G_0[P] \hookrightarrow G_0[\overline{P}]. $$ 
\end{proposition}

\begin{proof}
See the proofs of \cite[Lemma 2.4, Proposition 1.5]{fukasawa-miura1}. 
\end{proof}

\begin{corollary}
If $P$ is extendable quasi-Galois, then $\overline{P}$ is extendable quasi-Galois. 
\end{corollary}

\begin{corollary} 
Assume that $C$ is smooth and $d \ge 4$. 
If $R \in \mathbb P^{2*}$ is quasi-Galois for $C^*$, then $R$ is extendable and $|G[R]|=|G_0[R]| \le d$. 
\end{corollary}

\begin{proof}
According to \cite[Lemma 3.1]{fukasawa-miura1}, $R$ is extendable with $G_0[R]=G[R]$. 
By Proposition \ref{injection}, we have an injection $G_0[R] \hookrightarrow G_0[\overline{R}]$, where $\overline{R} \in \mathbb P^2$. 
Then, $|G_0[\overline{R}]| \le d$. 
\end{proof}

\begin{theorem} \label{dual}
If $P$ is extendable quasi-Galois, then $\overline{\overline{P}}=P$ and $G_0[P] \cong G_0[\overline{P}]$. 
\end{theorem}

\begin{proof}
Since $\overline{\overline{\sigma}}=\sigma$ for any $\sigma \in G_0[P]$, and $F[\overline{P}]$ does not depend on elements of $G_0[\overline{P}]\setminus \{1\}$ by Corollary \ref{fixed locus}, we have the conclusion. 
\end{proof}

\end{document}